\documentclass[a4paper,11pt,twoside]{article}
\RequirePackage[a4paper,head=1cm]{geometry}
\usepackage[utf8x]{inputenc}
\usepackage{textcomp}
\usepackage{amsmath,amsfonts,verbatim,afterpage,euscript,mathrsfs,amssymb,empheq}
\usepackage{amsthm}
\usepackage{dsfont}
\usepackage{hyperref}
\usepackage{pst-fill,pst-grad}
\usepackage{textcomp}
\usepackage{multicol}
\usepackage{amsmath}
\usepackage{amssymb}
\usepackage{hyperref} 
\usepackage{dsfont}

\newtheorem{Theoreme}{Theorem}
\newtheorem{Lemme}{Lemma}[section]
\newtheorem{Corollaire}{Corollary}[section]

\newtheorem{Remarque}{\bf Remark}
\newcommand{\mysection}{\setcounter{equation}{0} \section}

\def\vu{\vec{u}}

\title{ \bf On a pointwise estimate for a weighted rough singular integral operator in stratified Lie groups and applications} 
\author{Diego Chamorro\footnote{Laboratoire de Math\'ematiques et Mod\'elisation d'Evry (LaMME) - UMR 8071. Universit\'e d'Evry Val d'Essonne, 23 Boulevard de France, 91037 Evry Cedex, France. email (corresponding author): \textit{diego.chamorro@univ-evry.fr}},   Oscar Jarr\'in\footnote{Escuela de Ciencias F\'isicas y Matem\'aticas, Universidad de las Am\'ericas, V\'ia a Nay\'on, C.P.170124, Quito, Ecuador. email: \textit{oscar.jarrin@udla.edu.ec}}.} 
\begin{document} 
%%%%%%%%%%%%%%%%%%%%%%%%%%%%%%%%%%%%%%%%%%%%%%%%%%%
\maketitle 
\begin{scriptsize}
\abstract{ \noindent  In this article, we present a new pointwise estimate for a weighted rough singular integral operator in the setting of stratified Lie groups. This operator, $T_{\Omega, \varpi}$, is based on a kernel $\Omega$ and a weight $\varpi$, where the kernel satisfies a natural size condition and a cancellation property with respect to the weight $\varpi$. Moreover, we do not assume any kind of regularity on these objects. This weighted rough singular integral operator, applied to a function $f$, is estimated through a combination of information involving a weighted maximal function of the gradient of $f$ and a weighted Morrey space. We also deduce from this pointwise estimate some new weighted functional inequalities and, as an application, we obtain a uniqueness result for a rough version of the stationary Navier-Stokes equation over the Heisenberg group.}\\

{\footnotesize
\noindent \textbf{MSC (2020) Primary: 42B20; Secondary: 42B25}
}
\end{scriptsize}
%%%%%%%%%%%%%%%%%%%%%%%%%%%%%%%%%%%%%%%%%%%%%%%%%%%
%\tableofcontents 
%%%%%%%%%%%%%%%%%%%%%%%%%%%%%%%%%%%%%%%%%%%%%%%%%%%
\mysection{Introduction and presentation of the results}
In this article, we are interested in establishing a new pointwise control for \emph{weighted} rough singular integral operators in the general framework of stratified Lie groups $\mathbb{G}$. Many recent works, ranging from harmonic analysis to PDEs, take place within the setting of stratified Lie groups (see, \emph{e.g.}, \cite{BFG}, \cite{Cowling}, \cite{HanHuang}, \cite{Oka}, \cite{Tendani}), and these previous works motivate this exploration of a pointwise control for a generic class of singular operators over stratified Lie groups, especially since a pointwise estimate is a powerful tool that can lead to new developments. Stratified Lie groups are special examples of homogeneous spaces considered by Coifman and Weiss in \cite{CW}. They can be endowed with a dilation structure $\delta$ (which gives rise to the concept of \emph{homogeneous dimension} $N$) and with a distance $d$, providing them with a quite natural structure (see \cite{Folland2}; see also Section \ref{Secc_StratLieGroups} for more details about these groups) in which it is interesting to investigate new estimates for rough singular integral operators. These operators are of the following form
\begin{equation}\label{Def_Operator}
T_{\Omega, \varpi}(f)(x)= p.v.\int_{\mathbb{G}}\Omega(x,y)f(y)\varpi(y)dy,
\end{equation}
where $f:\mathbb{G}\longrightarrow \mathbb{R}$ is a suitable function (say $L^1_{loc}(\mathbb{G})$), and where the kernel $\Omega:\mathbb{G}\times \mathbb{G}\longrightarrow \mathbb{R}$ and the weight $\varpi:\mathbb{G}\longrightarrow  ]0,+\infty[$ satisfy some general conditions. For the kernel $\Omega$, the first condition is given by the following control: 
\begin{equation}\label{TailleKernel}
|\Omega(x,y)|\leq \frac{C_\Omega}{d(x,y)^N},
\end{equation}
 for all $x\neq y$ and where $N$ is the homogeneous dimension of the group. The second condition is the following weighted cancellation property 
\begin{equation}\label{IntegraleNulle}
\int_{\{a<d(x,y)<b\}}\Omega(x,y)\varpi(y)dy=0, 
\end{equation}
valid for all $x\in \mathbb{G}$ and for all $0<a<b<+\infty$. Note in particular that no regularity condition is asked for the kernel $\Omega$.\\
 
 \noindent Now, for the weight $\varpi$, we will first consider that it is doubling with respect of the Lebesgue measure of the group $\mathbb{G}$, \emph{i.e.} we have the property 
$$\int_{B(x,2r)}\varpi(y)dy=\varpi(B(x,2r))\leq C\varpi(B(x,r))=C\int_{B(x,r)}\varpi(y)dy,$$
 for all $x\in \mathbb{G}$, all $r>0$ and where $B(x,r)=\{y\in \mathbb{G}: d(x,y)<r\}$ is the usual metric ball associated to the distance $d$. Moreover we will also need the following upper Ahlfors condition for the weight $\varpi$:
\begin{equation}\label{Upper_Ahlfors_condition}
\varpi(B(x, r))\leq C r^{\nu}, \quad \mbox{ with } \quad 0<\nu<+\infty,
\end{equation}
for all $x\in \mathbb{G}$ and $r>0$.\\

As we can see, the setting for the kernel $\Omega$ and the weight $\varpi$ is quite general, and we emphasize that no regularity condition is required either for the kernel $\Omega$ or for the weight $\varpi$.\\

Finally, in the particular case when $\varpi\equiv 1$, \emph{i.e.}, in the unweighted setting, we will use the simpler notation $T_{\Omega}$ instead of $T_{\Omega,1}$.\\

To continue, and in order to state our main result, we need to introduce some (weighted) ingredients. First, we recall that a generic weight $\omega:\mathbb{G}\longrightarrow ]0,+\infty[$ is a locally integrable function, and we say that a weight $\omega$ belongs to the Muckenhoupt class $A_\sigma$ for some $1< \sigma<+\infty$ if
$$[\omega]_{A_{\sigma}}=\underset{B}{\sup}\left(\frac{1}{|B|}\int_{B}\omega(x)dx\right)\left(\frac{1}{|B|}\int_{B}\omega(x)^{-\frac{1}{\sigma-1}}dx\right)^{\sigma-1}<+\infty.$$
The Muckenhoupt class is one of the most important classes of weights when dealing with harmonic analysis tools, in particular in connection with the boundedness of maximal functions; see \cite[Chapter 8]{Grafakos}.\\ 

With this concept at our disposal, we will consider the following weighted tools: 
%%%%%%%%%%%%%%%%%%%%%%%%%%%%%%%%%%%%%%%%%%%%%%%%%%%
\begin{itemize}
\item[$\bullet$] {\bf Weighted maximal functions}. Let $\omega\in A_\sigma$ be a Muckenhoupt weight for some index $1<\sigma<+\infty$. Then, for a locally integrable function $f:\mathbb{G}\longrightarrow \mathbb{R}$, we define the weighted maximal function $\mathscr{M}_{\omega}$ of $f$ by
\begin{equation}\label{Weightedmaximalfunctions}
\mathscr{M}_{\omega}(f)(x)=\displaystyle{\underset{B \ni x}{\sup } \;\frac{1}{\omega(B)}\int_{B }|f(y)|\omega(y)dy}.
\end{equation}

\item[$\bullet$] {\bf Weighted Morrey spaces}. Associated with a Muckenhoupt weight $\omega\in A_\sigma$ with $1<\sigma<+\infty$, we define the weighted Morrey spaces $\mathcal{M}^{p,q}_\omega(\mathbb{G})$ with $1< p\leq q<+\infty$ by the condition
\begin{equation}\label{Def_Weighted_Morrey_space}
\|f\|_{\mathcal{M}^{p,q}_{\omega}}=\underset{x\in \mathbb{R}^n, \; r>0}{\sup}\left(\frac{1}{\omega(B(x,r))^{1-\frac{p}{q}}}\int_{B(x,r)}|f(y)|^p \omega(y)dy\right)^{\frac{1}{p}}<+\infty.
\end{equation}
Note that when $p=q$, the weighted Morrey space $\mathcal{M}^{q,q}_\omega(\mathbb{G})$ becomes the weighted Lebesgue space $L^{q}_\omega(\mathbb{G})$. Recall also that we have the space inclusion $L^{q}_\omega(\mathbb{G})\subset \mathcal{M}^{p,q}_\omega(\mathbb{G})$ for all $1<p\leq q$. 
\end{itemize}
\noindent See Section \ref{Secc_FuncSpaces} below for more details on weighted maximal functions $\mathscr{M}_{\omega}$, weighted Lebesgue spaces $L^q_\omega$, and weighted Morrey spaces $\mathcal{M}^{p,q}_\omega$, as well as some important boundedness properties of the maximal functions in this setting.\\

When no weight is considered (\emph{i.e.}, when $\omega\equiv 1$), we will simply write $\mathscr{M}$ and $\mathcal{M}^{p,q}$ to denote the usual maximal function and the usual Morrey spaces, respectively.\\

We can now state our main result: 
%%%%%%%%%%%%%%%%%%%%%%%%%%%%%%%%%%%%%%%%%%%%%%%%%%%
\begin{Theoreme}[{\bf A pointwise inequality}]\label{Theo1}
	Let $\mathbb{G}$ be a stratified Lie group with homogeneous dimension $N\geq 4$. Consider the operator $T_{\Omega, \varpi}$ defined in (\ref{Def_Operator}), where the kernel $\Omega$ satisfies the conditions (\ref{TailleKernel}) and (\ref{IntegraleNulle}), and where the weight $\varpi$ satisfies the upper Ahlfors condition (\ref{Upper_Ahlfors_condition}) with a power $0<\nu<+\infty$ such that $N-1<\nu$.\\
	
	\noindent Consider now a Muckenhoupt weight $\omega\in A_p$ with $1< p <+\infty$ that satisfies the following lower Ahlfors condition with a power $0<\rho<+\infty$:
	\begin{equation}\label{Lower_Ahlfors_condition}
	Cr^{\rho}\leq  C\omega(B(x,r)),\qquad \mbox{for all } 0<r<+\infty.
	\end{equation}
	
	\noindent Associated with this weight $\omega$, consider now a function $f:\mathbb{G}\longrightarrow \mathbb{R}$ such that its gradient $\nabla f$ belongs to the weighted Morrey space $\mathcal{M}^{p,q}_\omega(\mathbb{G})$ with $1< p< q<+\infty$.\\
	
\noindent Assume that the upper and lower Ahlfors powers $\nu, \rho$ of the weights $\varpi$ and $\omega$, the parameter $q$ of the weighted Morrey space $\mathcal{M}^{p,q}_\omega(\mathbb{G})$, and the homogeneous dimension $N$ are related by the condition
\begin{equation}\label{HypothesePointwise}
\nu+1-N-\frac{\rho}{q}<0,
\end{equation}
and that the weights $\varpi$ and $\omega$ satisfy the balance condition
\begin{equation}\label{HypotheseBalance}
\frac{r(B)}{r(B_0)}\left(\frac{\varpi(B)}{\varpi(B_0)}\right)^{\frac{1}{q}}\leq C \left(\frac{\omega(B_0)}{\omega(B)}\right)^{\frac{1}{p}},
\end{equation}
for all metric balls $B\subset cB_0$, where $r(B)$ denotes the radius of the ball, and $c,C$ are large constants. \\
	
	\noindent Then we have the following pointwise estimate:
	\begin{equation}\label{PointWiseIneqFeli1}
	T_{\Omega, \varpi}(f)(x)\leq C\left(\mathscr{M}_{\omega}(|\nabla f|^p)^{\frac{1}{p}}(x) \right)^{1-\theta}\|\nabla f\|_{\mathcal{M}_\omega^{p,q}}^\theta,
	\end{equation}
	where the index $0<\theta<1$ is given by $\theta=\frac{q}{\rho}(\nu+1-N)$.
\end{Theoreme}%%%%%%%%%%%%%%%%%%%%%%%%%%%%%%%%%%%%%%%%%%%%%%%%%%%
Some remarks are in order here. First, note that the properties of the operator $T_{\Omega, \varpi}$ depend heavily on the behavior of the kernel $\Omega$ and the weight $\varpi$. Indeed, the condition (\ref{TailleKernel}) on the kernel $\Omega$ is rather classical in Calder\'on-Zygmund theory (see Chapter 8 of \cite{Grafakos} and the references therein), since it provides sufficient control over the singular integral operator. Note, however, that although a cancellation condition of the type $\displaystyle{\int_{\mathbb{G}}\Omega(x,y)dy}=0$ is also classical, here we assume the condition (\ref{IntegraleNulle}), which establishes a deep link between the kernel $\Omega$ and the weight $\varpi$ and which turns out to be crucial for our computations. Note also that, in sharp contrast to the classical Calder\'on-Zygmund theory, no regularity condition is imposed on the kernel $\Omega$. Regarding the properties of the weight $\varpi$, we only assume that the measure $\varpi(y) dy$ is doubling, as well as satisfying the upper Ahlfors condition (\ref{Upper_Ahlfors_condition}), and we do not require it to be a Muckenhoupt weight. All these conditions provide a quite general framework for the operator $T_{\Omega, \varpi}$.\\

Second, we remark that, assuming some regularity for a function $f$ (namely, that $\nabla f$ belongs to a weighted Morrey space), the conclusion of the previous theorem provides the pointwise estimate (\ref{PointWiseIneqFeli1}). Note that this estimate also involves a weighted maximal function of $\nabla f$. However, this is not a very restrictive constraint, since maximal functions are generally bounded on most classical function spaces. Indeed, since we are able to obtain a suitable \emph{pointwise} control, several functional inequalities can be easily deduced, as we will see in Section \ref{Secc_FuncIneq}.\\

Third, we remark that the behavior of the weights $\varpi$ and $\omega$ is differentiated. Indeed, while the weight $\varpi$ satisfies the \emph{upper} Ahlfors condition given in (\ref{Upper_Ahlfors_condition}), the weight $\omega$ satisfies the \emph{lower} Ahlfors condition (\ref{Lower_Ahlfors_condition}). This distinction between the properties of the two weights can be contrasted with the assumptions made in the previous article \cite{ChMarcociMarcoci2}, where the measures involved satisfy both the upper \emph{and} the lower Ahlfors conditions. Note, however, that for technical reasons (see Lemma \ref{Lem_Poincare2poids} below), we will need the balance condition (\ref{HypotheseBalance}), which connects the weight $\varpi$ with the weight $\omega$.\\

Fourth, the pointwise estimate (\ref{PointWiseIneqFeli1}) can be connected to the sub-representation formulas studied in \cite{Hoang2} and to the pointwise estimates developed in \cite{ChMarcociMarcoci1}, \cite{Hoang}, \cite{Hoang1}, and \cite{Li}. However, the operator (\ref{Def_Operator}) considered here is quite different from the operators considered in those works, since the operator $T_{\Omega, \varpi}$ is built from a kernel $\Omega$ and a weight $\varpi$. In this sense, the results presented here are, to the best of our knowledge, new even in the Euclidean framework.\\

Finally, let us remark that the condition $N\geq 4$ is essentially related to the fact that one of the simplest, yet nontrivial, stratified Lie groups is the Heisenberg group $\mathbb{H}$, which is based on $\mathbb{R}^3$ but whose homogeneous dimension is $4$. Note also that the condition $\nu>N-1$ imposed on the power $\nu$, as well as the constraint (\ref{HypothesePointwise}) relating the parameters $N, \nu, \rho$, and $q$, are essentially technical, and we do not claim any optimality of these parameters. However, a different range of values for these parameters would probably require different techniques that are beyond the scope of this article.\\

To close this section, let us point out that, to the best of our knowledge, the study of this type of weighted singular integral operator in the setting of stratified Lie groups is new. \\
 
{\bf Organization of the article}. In Section \ref{Secc_StratLieGroups}, we present the basic properties of stratified Lie groups, while in Section \ref{Secc_FuncSpaces}, we recall the definitions and main properties of the function spaces used here in connection with Muckenhoupt weights. In Section \ref{Secc_Theo1}, we prove the main Theorem \ref{Theo1}, and Sections \ref{Secc_FuncIneq} and \ref{Secc_Appl} are devoted to some applications of the pointwise estimate (\ref{PointWiseIneqFeli1}). In particular, we study the uniqueness of the trivial solution for a rough version (\emph{i.e.}, with a rough drift) of the stationary Navier-Stokes equations over the Heisenberg group.
%%%%%%%%%%%%%%%%%%%%%%%%%%%%%%%%%%%%%%%%%%%%%%%%%%%
\mysection{Stratified Lie groups}\label{Secc_StratLieGroups}
Stratified Lie groups are natural generalizations of $\mathbb{R}^n$ when considering general dilation structures. Although stratified Lie groups share common features with $\mathbb{R}^n$, there are some important differences that must be taken into account. For example, these groups are no longer abelian, and this fact requires care in some computations. Furthermore, from a geometric point of view, the underlying geometric structure of these groups can be very different from the Euclidean setting. It is therefore necessary to recall some basic facts about stratified Lie groups. For further information, see \cite{Folland2}, \cite{Stein2}, \cite{Varopoulos}, and the references therein. \\

We start with the notion of a \textit{homogeneous group} $\mathbb{G}$, which consists of $\mathbb{R}^{n}$ equipped with a Lie group structure, and we will always suppose that the origin is the identity. We define a \textit{dilation structure} by fixing integers $(a_{i})_{1\leq i\leq n}$ such that $1=a_{1}\leq... \leq a_{n}$ and by writing:
\begin{eqnarray}
\delta_{\alpha }:  \mathbb{R}^{n } & \longrightarrow & \mathbb{R}^{n } \label{dilat} \\
x &\longmapsto & \delta_{\alpha}[x]=(\alpha^{a_{1}}x_{1},...,\alpha^{a_{n}}x_{n}).\nonumber
\end{eqnarray}
We will often write $\alpha x$ instead of $\delta_{\alpha}[x]$, and $\alpha$ will always indicate a strictly positive real number. These dilations are, of course, group automorphisms. The Euclidean space $\mathbb{R}^{n}$, with its usual group structure and equipped with its usual dilations (i.e. $a_{i}=1$, for $i=1,...,n$), is a homogeneous group. Here is another example: if $x=(x_{1}, x_{2}, x_{3})$ is an element of $\mathbb{R}^{3}$, we can define a dilation by writing $\delta_{\alpha}[x]=(\alpha x_{1 }, \alpha x_{2}, \alpha^{2 } x_{3})$ for $\alpha>0$. Then, the well-suited group law with respect to this dilation is given by $x\cdot y=(x_{1}, x_{2}, x_{3})\cdot(y_{1}, y_{2}, y_{3})=(x_{1}+y_{1}, x_{2}+y_{2}, x_{3}+y_{3}+\frac{1}{2}(x_{1}y_{2}-y_{1}x_{2 }))$. Note in particular that this group law is no longer abelian. The triplet $(\mathbb{R}^{3}, \cdot,\delta )$ corresponds to the Heisenberg group $\mathbb{H}$, which is the first non-trivial example of a homogeneous group.\\ 

The \emph{homogeneous dimension} with respect to the dilation structure (\ref{dilat}) is given by
$$N=\displaystyle{\sum_{1\leq i\leq n}}a_{i }.$$
We observe that it is always greater than or equal to the topological dimension $n$, since each integer $a_{i}$ satisfies $a_{i}\geq 1$ for all $i=1,...,n$. For instance, in the Heisenberg group $\mathbb{H}^{1}$, we have $N=4$ and $n=3$, while in the Euclidean case, these two concepts coincide.\\

%%%%%%%%%%%%%%%%%%%%%%%%%%%%%%%%%%%%%%%%%%%%%%%%%%%
From the point of view of measure theory, homogeneous groups behave in a traditional way, since Lebesgue measure $dx$ is bi-invariant and coincides with the Haar measure. Thus, for any measurable subset $A$ of $\mathbb{G}$, we will denote its measure by $|A|$.\\

%%%%%%%%%%%%%%%%%%%%%%%%%%%%%%%%%%%%%%%%%%%%%%%%%%%
For a homogeneous group $\mathbb{G}=(\mathbb{R}^{n}, \cdot, \delta)$, we consider its Lie algebra $\mathfrak{g}$, whose elements can be viewed in two different ways: as \textit{left}-invariant vector fields or as \textit{right}-invariant vector fields. The left-invariant vector fields $(X_j)_{1\leq j\leq n}$ are determined by the formula
\begin{equation*}
(X_{j}f)(x)=\left.\frac{\partial f(x\cdot y)}{\partial y_{j}}\right|_{y=0}=\frac{\partial f}{\partial x_{j}}+\sum_{j<k}q^{k}_{j}(x)\frac{\partial f}{\partial x_{k}},
\end{equation*}
where $q^{k}_{j}(x)$ is a homogeneous polynomial of degree $a_{k}-a_{j}$ and $f$ is a smooth function on $\mathbb{G}$. From this formula, one easily deduces that these vector fields are homogeneous of degree $a_{j}$, and we have $X_{j}\left(f(\alpha x)\right)=\alpha^{a_{j}}(X_{j}f)(\alpha x)$. \\

\noindent Now, a homogeneous group $\mathbb{G}$ is \emph{stratified} if its Lie algebra $\mathfrak{g}$ can be decomposed as a sum of linear subspaces
$\mathfrak{g}=\bigoplus_{1\leq j\leq k } E_{j}$ such that $E_{1}$ generates the algebra $\mathfrak{g}$ and $[E_{1}, E_{j}]=E_{j+1}$ for $1\leq j < k$, $[E_{1}, E_{k}]=\{0\}$, and $E_{k}\neq\{0\}$, while $E_{j}=\{0\}$ if $j>k$. Here, $[E_{1}, E_{j}]$ indicates the subspace of $\mathfrak{g}$ generated by the elements $[U, V]=UV-VU$ with $U\in E_{1}$ and $V\in E_{j}$. The integer $k$ is called the \emph{degree of stratification} of $\mathfrak{g}$. For example, on the Heisenberg group $\mathbb{H}$, we have $k=2$, while in the Euclidean case, $k=1$.\\

We will suppose from now on that $\mathbb{G}$ is \textbf{stratified} with homogeneous dimension\footnote{The lower bound $N\geq 4$ corresponds to the homogeneous dimension of the Heisenberg group $\mathbb{H}$, which is the simplest non-trivial stratified Lie group.} $N\geq 4$. Within this framework, we will fix once and for all the family of vector fields
\begin{equation}\label{Hormander}
{\bf X}=\{X_1,...,X_m\},
\end{equation}
such that $a_{1}=a_{2}=\ldots=a_{m}=1$ $(m<n)$. Then, the family $\textbf{X}$ is a basis of $E_{1}$ and generates the Lie algebra $\mathfrak{g}$, which is precisely H\"ormander's condition (see \cite{Folland2} and \cite{Varopoulos}). This particular choice ensures several important properties. In particular, the family $\textbf{X}$ is associated with the Carnot-Carath\'eodory distance $d$, which is left-invariant and compatible with the topology on $\mathbb{G}$ (see \cite{Varopoulos} for more details). Thus, for any $x\in \mathbb{G}$ and $r>0$, we can define open balls by
\begin{equation}\label{DefBall}
B(x,r)=\{y\in \mathbb{G}: d(x,y)<r\}.
\end{equation}
We remark that, by simple homogeneity arguments, stratified Lie groups have polynomial volume growth, since we have $|B(\cdot, r)|=r^N |B(\cdot, 1)|$.\\

We define now the \textit{gradient} on $\mathbb{G}$ using vector fields of homogeneity degree equal to one (\textit{i.e.}, those composing the family $\textbf{X}$), by setting $\nabla = (X_{1},...,X_{m})$. This operator is, of course, left-invariant and homogeneous of degree $1$. The length of the gradient is given by the formula
$$|\nabla f|= \left((X_{1}f)^{2}+... +(X_{m}f)^{2 } \right)^{1/2}.$$

For more details concerning stratified Lie groups see the books \cite{Folland2},  \cite{Stein2}, \cite{Varopoulos} and the articles \cite{Folland0}, \cite{Saka} as well as the references therein.
%%%%%%%%%%%%%%%%%%%%%%%%%%%%%%%%%%%%%%%%%%%%%%%%%%%
\mysection{Functional spaces and some inequalities}\label{Secc_FuncSpaces}
In this section, we recall the definitions and some well-known properties of the function spaces, Muckenhoupt weights, and maximal function operators that will be used in the sequel.
%%%%%%%%%%%%%%%%%%%%%%%%%%%%%%%%%%%%%%%%%%%%%%%%%%%
%%%%%%%%%%%%%%%%%%%%%%%%%%%%%%%%%%%%%%%%%%%%%%%%%%%
\subsubsection*{Weighted functional spaces}
We now consider weighted function spaces and recall some of their properties. 
\begin{itemize}
\item {\bf Weighted Lebesgue spaces}. For $1\leq q<+\infty$, for a weight $\omega \in A_\sigma$ with $1< \sigma<+\infty$, and for $A=\mathbb{G}$ or a measurable subset $A \subset \mathbb{G}$, the weighted Lebesgue spaces $L^q_\omega(A)$ are defined by the condition
$$\|f\|_{L^q_\omega}=\left(\int_{A}|f(x)|^q\omega(x)dx\right)^{\frac{1}{q}}<+\infty.$$
Note that, for some parameter $1<\alpha<+\infty$, we have the property
\begin{equation}\label{ProprietePuissanceLebesgue}
\||f|^\alpha\|_{L^q_\omega}=\|f\|_{L^{\alpha q}_\omega}^\alpha.
\end{equation}
We also remark that, in this weighted setting, we have the following H\"older inequality for $\frac{1}{q}+\frac{1}{q'}=1$, for some measurable set $A\subseteq\mathbb{G}$ and for $\omega\in A_\sigma$ with $1\leq \sigma<+\infty$:
$$\int_{A}|f(x)|\omega(x)dx\leq \|f\|_{L^q_\omega(A)}\|f\|_{L^{q'}_\omega(A)},$$
and note that this inequality also holds for generic weights.\\

We recall now that the weighted maximal function $\mathscr{M}_\omega$ given in (\ref{Weightedmaximalfunctions}) is bounded on the weighted Lebesgue spaces, and we have
\begin{equation}\label{BoundednessMaxWeightedFuncLebesgue}
\|\mathscr{M}_\omega(f)\|_{L^q_\omega}\leq C\|f\|_{L^q_\omega},
\end{equation}
where $1<q<+\infty$ and $\omega\in A_\sigma$ for $1< \sigma <+\infty$. See a proof of this fact in \cite[Theorem 2.6, p. 146]{GCRdF}, \cite{Pan} for a proof of this estimate in the more general setting of spaces of homogeneous type, and also \cite[Theorem 3.1]{Komori}. 
%%%%%%%%%%%%%%%%%%%%%%%%%%%%%%%%%%%%%%%%%%%%%%%%%%%
\item {\bf Weighted Morrey spaces}. The spaces $\mathcal{M}_\omega^{p,q}(\mathbb{G})$ defined in expression (\ref{Def_Weighted_Morrey_space}) satisfy the following useful property:
\begin{equation}\label{ProprietePuissanceMorrey}
\||f|^\alpha\|_{\mathcal{M}_\omega^{p,q}}=\|f\|_{\mathcal{M}_\omega^{\alpha p, \alpha q}}^\alpha.
\end{equation}
We also have the estimate
$$\|f\|_{\mathcal{M}_\omega^{p_1,q}}\leq \|f\|_{\mathcal{M}_\omega^{p_0,q}},$$
for all $1<p_0\leq p_1<q$, which gives the space inclusion $\mathcal{M}_\omega^{p_0,q}(\mathbb{G})\subset \mathcal{M}_\omega^{p_1,q}(\mathbb{G})$. Moreover, following\footnote{The proof presented in \cite{Komori} can be easily adapted to the setting of stratified Lie groups studied in this article since the weight $\omega$ considered here is doubling.} \cite[Theorem 3.1]{Komori}, the maximal function $\mathscr{M}_\omega$ is also bounded on these weighted Morrey spaces, and we have
\begin{equation}\label{BoundednessMaxWeightedFuncMorrey}
\|\mathscr{M}_\omega(f)\|_{\mathcal{M}_\omega^{p,q}}\leq C\|f\|_{\mathcal{M}_\omega^{p,q}},
\end{equation}
where $\omega \in A_\sigma$ for $1< \sigma<+\infty$ and where $1<p\leq q<+\infty$. 
\end{itemize}
%%%%%%%%%%%%%%%%%%%%%%%%%%%%%%%%%%%%%%%%%%%%%%%%%%%
\subsubsection*{A double weight inequality}
We need to recall an important inequality that will be crucial in the sequel. 
%%%%%%%%%%%%%%%%%%%%%%%%%%%%%%%%%%%%%%%%%%%%%%%%%%%
\begin{Lemme}\label{Lem_Poincare2poids}
For a function $f\in \mathcal{C}^\infty(\mathbb{R}^n)$ and for every ball $B(x,r)$ such that $B(x,r)\subset supp(f)$, we have the following \emph{Poincaré-Sobolev inequality}:
\begin{equation}\label{PoincareSobolev_inequality}
\left(\frac{1}{\varpi(B(x,r))}\int_{B(x,r)}|f(y)-f_{B_r}|^{q}\varpi(y)dy\right)^{\frac{1}{q}}\leq Cr\left(\frac{1}{\omega(B(x,r))}\int_{B(x,r)}|\nabla f(y)|^{p}\omega(y)dy\right)^{\frac{1}{p}}. 
\end{equation}
for $1<p<q<+\infty$, where the weights $\omega$ and $\varpi$ satisfy the following \emph{balance condition}:
\begin{equation}
\frac{r(B)}{r(B_0)}\left(\frac{\varpi(B)}{\varpi(B_0)}\right)^{\frac{1}{q}}\leq C \left(\frac{\omega(B_0)}{\omega(B)}\right)^{\frac{1}{p}},
\end{equation}
for all metric balls $B\subset cB_0$, where $c,C$ are large constants.
\end{Lemme}
\noindent See \cite[Theorem 1.1]{HanHuang}, \cite{LuWheeded} and the references therein for more details about this inequality.
%%%%%%%%%%%%%%%%%%%%%%%%%%%%%%%%%%%%%%%%%%%%%%%%%%%
\begin{Remarque}\label{PoinKSansPoids}
Let us note that when $\varpi=\omega=1$ (\emph{i.e.}, when no weights are considered), we have the following version of the inequality (\ref{PoincareSobolev_inequality}):
\begin{equation}\label{PoincareSobolev_inequality1}
\left(\frac{1}{|B(x,r)|}\int_{B(x,r)}|f(y)-f_{B_r}|^{q}dy\right)^{\frac{1}{q}}\leq Cr\left(\frac{1}{|B(x,r)|}\int_{B(x,r)}|\nabla f(y)|^{p}dy\right)^{\frac{1}{p}},
\end{equation}
where $1<p<N$ and $q=\frac{Np}{N-p}$. 
\end{Remarque}
\noindent See \cite{Lu} for a proof of this fact, as well as the references therein, for more details on this unweighted inequality in the setting of stratified Lie groups.\\
%%%%%%%%%%%%%%%%%%%%%%%%%%%%%%%%%%%%%%%%%%%%%%%%%%%

To end this section, we recall some results in the unweighted setting that will be used in Section \ref{Secc_Appl} below. To start with, we recall that, associated with a H\"ormander family ${\bf X}$ fixed in (\ref{Hormander}), we can define a sub-Laplacian by $\Delta=\sum_{j=1}^mX_j^2$ and its associated semigroup $H_t=e^{-t\Delta}$, which admits a convolution kernel $h_t$. We now define, for $s>0$, the fractional powers of this operator by the formulas
$$(-\Delta)^{\frac{s}{2}}(f)=\frac{1}{\Gamma(k-\frac{s}{2})}\int_{0}^{+\infty}t^{k-\frac{s}{2}-1}(-\Delta)^{k}H_{t}(f)dt\qquad\mbox{and}\qquad (-\Delta)^{\frac{-s}{2}}(f)=\frac{1}{\Gamma(s)}\int_{0}^{+\infty}t^{\frac{s}{2}-1}H_{t}(f)dt,$$
where $k$ is the smallest integer such that $k>\frac{s}{2}$.\\

For $s\in \mathbb{R}$ and $1<p<+\infty$, the homogeneous Sobolev spaces $\dot{W}^{s,p}(\mathbb{G})$ are now defined by the condition
$$\|f\|_{\dot{W}^{s,p}}=\|(-\Delta)^{\frac{s}{2}}(f)\|_{L^p}<+\infty.$$
Note that when $s$ is an integer, then we have the following identification 
$$\sum_{|I|=s}\|X^{I}(f)\|_{L^p}\simeq \|(-\Delta)^{\frac{|I|}{2}}(f)\|_{L^p},$$
where $X^I=X_1^{i_1}...X_{m}^{i_m}$ and $I = (i_1, ..., i_m)$ is a multi-index. \\

Recall that we have the classical Sobolev inequalities given by the estimates
\begin{equation}\label{SobolevIneq}
\|f\|_{L^q}\leq C\|\nabla f\|_{L^p},
\end{equation}
where $1<p<N$ and $q=\frac{Np}{N-p}$, as well as the Hardy-Littlewood-Sobolev inequalities
\begin{equation}\label{HLSobolevIneq}
\|(-\Delta)^{\frac{-1}{2}}(f)\|_{L^q}\leq C\|f\|_{L^p},
\end{equation}
where $1<p<N$ and $q=\frac{Np}{N-p}$. To finish, we point out that the Riesz transforms $(-\Delta)^{-\frac{1}{2}}X_j$ are bounded in the Lebesgue spaces as we have the control 
\begin{equation}\label{RieszTransforms}
\|(-\Delta)^{-\frac{1}{2}}X_j(f)\|_{L^p}\leq C\|f\|_{L^p},
\end{equation}
for all $1<p<+\infty$.\\

For more details about functional spaces and harmonic analysis results, see \cite{Folland0}, \cite{Folland2}, \cite{Saka}, \cite{Stein2},  \cite{Varopoulos} and the references therein.  
%%%%%%%%%%%%%%%%%%%%%%%%%%%%%%%%%%%%%%%%%%%%%%%%%%%
%%%%%%%%%%%%%%%%%%%%%%%%%%%%%%%%%%%%%%%%%%%%%%%%%%%
\mysection{Proof of  Theorem \ref{Theo1}}\label{Secc_Theo1}

For some fixed $t>0$, we consider the operator 
$$T^t_{\Omega, \varpi}(f)(x)=\int_{\{t<d(x,y)\}}\Omega(x,y)f(y)\varpi(y)dy,$$
and let us write 
$$T^*_{\Omega, \varpi}(f)(x)=\underset{t>0}{\sup}|T^t_{\Omega, \varpi}(f)(x)|=\underset{t>0}{\sup}\left|\int_{\{t<d(x,y)\}}\Omega(x,y)f(y)\varpi(y)dy\right|.$$
Note that with this notation, we have the control 
$$T_{\Omega, \varpi}(f)(x)\leq |T_{\Omega, \varpi}(f)(x)|\leq T^*_{\Omega, \varpi}(f)(x),$$ 
and thus we will focus our study to obtain an uniform estimate (with respect to the parameter $t>0$) of the operator $T^t_{\Omega, \varpi}$.\\

Now, for a function $f\in L^1_{loc}(\mathbb{G})$, regular enough, and for some $k_0\in \mathbb{Z}$ so that $2^{k_0-2}<t \leq 2^{k_0-1}$, we write
$$T^t_{\Omega, \varpi}(f)(x)=\int_{\{t<d(x,y)\leq 2^{k_0-1}\}}\Omega(x,y)f(y)\varpi(y)dy+\sum_{k\geq k_0}\int_{\{2^{k-1}<d(x,y)\leq 2^{k}\}}\Omega(x,y)f(y)\varpi(y)\varpi(y)dy.$$
Since we have the cancellation condition $\displaystyle{\int_{\{a<d(x,y)<b\}}}\Omega(x,y)\varpi(y)dy=0$ given in (\ref{IntegraleNulle}), we can introduce some constants, that we define later on, in the previous expression to obtain
$$T^t_{\Omega, \varpi}(f)(x)=\int_{\{t<d(x,y)\leq 2^{k_0-1}\}}\Omega(x,y)(f(y)-c_{k_0})\varpi(y)dy+\sum_{k\geq k_0}\int_{\{2^{k-1}<d(x,y)\leq 2^{k}\}}\Omega(x,y)(f(y)-c_k)\varpi(y)dy,$$
from which we deduce the inequality
$$|T^t_{\Omega, \varpi}(f)(x)|\leq\sum_{k\in\mathbb{Z}}\int_{\{2^{k-1}<d(x,y)\leq 2^{k}\}}|\Omega(x,y)| |f(y)-c_k|\varpi(y)dy.$$
Recall that by the hypothesis (\ref{TailleKernel}) we have the control $|\Omega(x,y)|\leq \frac{C_\Omega}{d(x,y)^N}$ for all $x\neq y$, we can thus write (since we integrate over the set $\{2^{k-1}<d(x,y)\leq 2^{k}\}$)
\begin{eqnarray*}
|T^t_{\Omega, \varpi}(f)(x)|&\leq& C_\Omega\sum_{k\in\mathbb{Z}}\int_{\{2^{k-1}<d(x,y)\leq 2^{k}\}}\frac{1}{d(x,y)^N}{|f(y)-c_k|}\varpi(y)dy\\
&\leq &C_\Omega\sum_{k\in\mathbb{Z}}\int_{\{2^{k-1}<d(x,y)\leq 2^{k}\}}\frac{1}{2^{(k-1)N}}{|f(y)-c_k|}\varpi(y)dy,
\end{eqnarray*}
which we rewrite as follows 
$$|T^t_{\Omega, \varpi}(f)(x)|\leq C_\Omega\cdot 2^N\sum_{k\in\mathbb{Z}}\frac{1}{2^{kN}}  \int_{B(x,2^{k})}|f(y)-c_k|\varpi(y)dy.$$
We apply now the H\"older inequality in the weighted setting with $\frac{1}{q}+\frac{1}{q'}=1$ and $1<q<N$, to obtain
\begin{eqnarray*}
|T^t_{\Omega, \varpi}(f)(x)|&\leq &C_\Omega\cdot 2^N\sum_{k\in\mathbb{Z}}\frac{1}{2^{kN}}  \varpi(B(x, 2^k))^{\frac{1}{q'}}\left(\int_{B(x,2^{k})}|f(y)-c_k|^q\varpi(y)dy\right)^{\frac{1}{q}}\\
&\leq & C_\Omega\cdot 2^N\sum_{k\in\mathbb{Z}}\frac{1}{2^{kN}}  \varpi(B(x, 2^k))\left(\frac{1}{\varpi(B(x, 2^k))}\int_{B(x,2^{k})}|f(y)-c_k|^q\varpi(y)dy\right)^{\frac{1}{q}}.
\end{eqnarray*}
Assuming now the upper Ahlfors condition for the weight $\varpi$, we have $\varpi(B(x, 2^k))\leq C2^{k\nu}$ for some power $0<\nu<+\infty$ such that $N-1< \nu$ (recall the condition (\ref{Upper_Ahlfors_condition}) above), and thus we have
\begin{equation}\label{AvantPoinK}
|T^t_{\Omega, \varpi}(f)(x)|\leq  C \sum_{k\in\mathbb{Z}}\frac{1}{2^{kN}}  2^{k\nu}\left(\frac{1}{\varpi(B(x, 2^k))}\int_{B(x,2^{k})}|f(y)-c_k|^q\varpi(y)dy\right)^{\frac{1}{q}}.
\end{equation}
At this point we use the two-weights $(\varpi,\omega)$-Poincaré inequality given in the Lemma \ref{Lem_Poincare2poids}, to obtain 
\begin{eqnarray}\label{ApresPoinK}
|T^t_{\Omega, \varpi}(f)(x)|&\leq &C\sum_{k\in\mathbb{Z}}\frac{1}{2^{k(N-\nu)}}2^k\left(\frac{1}{\omega(B(x,2^{k}))} \int_{B(x, 2^{k})}|\nabla f(y)|^p\omega(y)dy\right)^{\frac{1}{p}}\\
&\leq & C\sum_{k\in\mathbb{Z}}\frac{1}{2^{k(N-1-\nu)}}\left(\frac{1}{\omega(B(x,2^{k}))} \int_{B(x, 2^{k})}|\nabla f(y)|^p\omega(y)dy\right)^{\frac{1}{p}},\notag
\end{eqnarray}
where $1<p<q<N$. We decompose now the previous sum in two parts: 
\begin{eqnarray*}
|T^t_{\Omega, \varpi}(f)(x)|&\leq & C\sum_{k\leq \lfloor \log_2(\mathcal{K})\rfloor}\frac{1}{2^{k(N-1-\nu)}}\left(\frac{1}{\omega(B(x,2^{k}))} \int_{B(x, 2^{k})}|\nabla f(y)|^p\omega(y)dy\right)^{\frac{1}{p}}\\
&&+ C\sum_{k\geq \lfloor \log_2(\mathcal{K})\rfloor}\frac{1}{2^{k(N-1-\nu)}}\left(\frac{1}{\omega(B(x,2^{k}))} \int_{B(x, 2^{k})}|\nabla f(y)|^p\omega(y)dy\right)^{\frac{1}{p}},
\end{eqnarray*}
where $0<\mathcal{K}<+\infty$ is a parameter that will be fixed later on and where $\lfloor\cdot \rfloor$ denotes the integer part of a real number.\\

Recall now that by hypothesis we have $N-1-\nu<0$ (or equivalently $\nu+1-N>0$) and we can write
\begin{eqnarray*}
|T^t_{\Omega, \varpi}(f)(x)|&\leq & C\underbrace{\sum_{k\leq \lfloor \log_2(\mathcal{K})\rfloor}2^{k(\nu+1-N)}\left(\frac{1}{\omega(B(x,2^{k}))} \int_{B(x, 2^{k})}|\nabla f(y)|^p\omega(y)dy\right)^{\frac{1}{p}}}_{(A)}\\
&&+ C\underbrace{\sum_{k\geq \lfloor \log_2(\mathcal{K})\rfloor}2^{k(\nu+1-N)}\left(\frac{1}{\omega(B(x,2^{k}))} \int_{B(x, 2^{k})}|\nabla f(y)|^p\omega(y)dy\right)^{\frac{1}{p}}}_{(B)},
\end{eqnarray*}
and we will study these terms separately. For the term $(A)$ above, we write 
\begin{eqnarray*}
(A)&=&\sum_{k\leq \lfloor \log_2(\mathcal{K})\rfloor}2^{k(\nu+1-N)}\left(\frac{1}{\omega(B(x,2^{k}))} \int_{B(x, 2^{k})}|\nabla f(y)|^p\omega(y)dy\right)^{\frac{1}{p}}\\
&\leq &\sum_{k\leq \lfloor \log_2(\mathcal{K})\rfloor}2^{k(\nu+1-N)}\mathscr{M}_{\omega}(|\nabla f|^p)^{\frac{1}{p}}(x),
\end{eqnarray*}
where we used the definition of the weighted maximal function $\mathscr{M}_{\omega}$ given in the expression (\ref{Weightedmaximalfunctions}) above. We obtain then 
$$(A)\leq \mathscr{M}_{\omega}(|\nabla f|^p)^{\frac{1}{p}}(x) \sum_{k\leq \lfloor \log_2(\mathcal{K})\rfloor}2^{k(\nu+1-N)},$$
and computing the sum, we have
$$(A)\leq C\mathscr{M}_{\omega}(|\nabla f|^p)^{\frac{1}{p}}(x) \mathcal{K}^{\nu+1-N}.$$
For the term $(B)$, we write 
\begin{eqnarray*}
(B)&=&\sum_{k\geq \lfloor \log_2(\mathcal{K})\rfloor}2^{k(\nu+1-N)}\left(\frac{1}{\omega(B(x,2^{k}))} \int_{B(x, 2^{k})}|\nabla f(y)|^p\omega(y)dy\right)^{\frac{1}{p}}\\
&=&\sum_{k\geq \lfloor \log_2(\mathcal{K})\rfloor}2^{k(\nu+1-N)}\left(\frac{\omega(B(x,2^{k}))^{-\frac{p}{q}}}{\omega(B(x,2^{k}))^{1-\frac{p}{q}}} \int_{B(x, 2^{k})}|\nabla f(y)|^p\omega(y)dy\right)^{\frac{1}{p}}\\
&=&\sum_{k\geq \lfloor \log_2(\mathcal{K})\rfloor}2^{k(\nu+1-N)}\omega(B(x,2^{k}))^{-\frac{1}{q}}\left(\frac{1}{\omega(B(x,2^{k}))^{1-\frac{p}{q}}} \int_{B(x, 2^{k})}|\nabla f(y)|^p\omega(y)dy\right)^{\frac{1}{p}}.
\end{eqnarray*}
We use now the lower Ahlfors property of the weight $\omega$, \emph{i.e.} $\omega(B(x,2^{k}))^{-\frac{1}{q}}\leq C (2^k)^{-\frac{\rho}{q}}$ (recall the formula (\ref{Lower_Ahlfors_condition}) above), to write
$$(B)\leq C\sum_{k\geq \lfloor \log_2(\mathcal{K})\rfloor}2^{k(\nu+1-N)}2^{-k\frac{\rho}{q}}\left(\frac{1}{\omega(B(x,2^{k}))^{1-\frac{p}{q}}} \int_{B(x, 2^{k})}|\nabla f(y)|^p\omega(y)dy\right)^{\frac{1}{p}},$$
and by the definition of the weighted Morrey spaces $\mathcal{M}_{\omega}^{p,q}(\mathbb{G})$, given in the formula (\ref{Def_Weighted_Morrey_space}), we have
\begin{eqnarray*}
(B)&\leq &C\sum_{k\geq \lfloor \log_2(\mathcal{K})\rfloor}2^{k(\nu+1-N-\frac{\rho}{q})}\|\nabla f \|_{\mathcal{M}_{\omega}^{p,q}}\\
&\leq & C\|\nabla f \|_{\mathcal{M}_{\omega}^{p,q}}\sum_{k\geq \lfloor \log_2(\mathcal{K})\rfloor}2^{k(\nu+1-N-\frac{\rho}{q})}.
\end{eqnarray*}
At this point we remark that by the hypothesis (\ref{HypothesePointwise}) we have $\nu+1-N-\frac{\rho}{q}<0$ , and thus the previous sum converges and we obtain
$$(B)\leq  C \|\nabla f \|_{\mathcal{M}_{\omega}^{p,q}}\mathcal{K}^{\nu+1-N-\frac{\rho}{q}}.$$
With the estimates for the terms $(A)$ and $(B)$, we have
$$|T^t_{\Omega, \varpi}(f)(x)|\leq C \mathscr{M}_{\omega}(|\nabla f|^p)^{\frac{1}{p}}(x)  \mathcal{K}^{\nu+1-N}+C \|\nabla f \|_{\mathcal{M}_{\omega}^{p,q}}\mathcal{K}^{\nu+1-N-\frac{\rho}{q}},$$
now we fix the parameter $\mathcal{K}$ by setting $\mathcal{K}=\left(\frac{\|\nabla f \|_{\mathcal{M}_{\omega}^{p,q}}}{\mathscr{M}_{\omega}(|\nabla f|^p)^{\frac{1}{p}}(x)}\right)^{\frac{q}{\rho}}$ and we obtain the estimate 
$$|T^t_{\Omega, \varpi}(f)(x)|\leq C\left(\mathscr{M}_{\omega}(|\nabla f|^p)^{\frac{1}{p}}(x)\right)^{1-\frac{q}{\rho}(\nu+1-N)} \|\nabla f \|_{\mathcal{M}_{\omega}^{p,q}}^{\frac{q}{\rho}(\nu+1-N)},$$
and this ends the proof of the Theorem \ref{Theo1} since we have $0<\theta=\frac{q}{\rho}(\nu+1-N)<1$ and since we have the uniform (with respect to the truncation parameter $t$) estimate
$T_{\Omega, \varpi}(f)(x)\leq |T^t_{\Omega, \varpi}(f)(x)|$. \hfill $\blacksquare$
%%%%%%%%%%%%%%%%%%%%%%%%%%%%%%%%%%%%%%%%%%%%%%%%%%%
\begin{Remarque}\label{InegaliteSansPoids}
Let us remark that if we set $\varpi=\omega=1$, \emph{i.e.} when an unweighted setting is considered, then we will have $\nu=\rho=N$ and the corresponding measure is the usual Lebesgue measure that satisfies the upper Ahlfors condition (\ref{Upper_Ahlfors_condition}) as well as the lower Ahlfors condition (\ref{Lower_Ahlfors_condition}).
\end{Remarque}
%%%%%%%%%%%%%%%%%%%%%%%%%%%%%%%%%%%%%%%%%%%%%%%%%%%
In this unweighted setting, we obtain the following pointwise estimate
%%%%%%%%%%%%%%%%%%%%%%%%%%%%%%%%%%%%%%%%%%%%%%%%%%%
\begin{Corollaire}[\bf Unweighted pointwise inequality]\label{CorollaireEstimationPonctuelle}
Let $\mathbb{G}$ be a stratified Lie group with homogeneous dimension $N\geq 4$. Consider the operator $T_{\Omega}$ defined in (\ref{Def_Operator}), where the kernel $\Omega$ satisfies the conditions (\ref{TailleKernel}) and (\ref{IntegraleNulle}) with $\varpi\equiv 1$.\\
 
\noindent Let $f:\mathbb{G}\longrightarrow \mathbb{R}$ be a function such that its gradient $\nabla f$ belongs to the Morrey space $\mathcal{M}^{p,q}(\mathbb{G})$ with $1<p<\frac{N}{2}$ and $q=\frac{Np}{N-p}$, then we obtain the following pointwise inequality
\begin{equation}\label{PointwiseUnweighted}
T_{\Omega}(f)(x)\leq C\left(\mathscr{M}(|\nabla f|^p)^{\frac{1}{p}}(x) \right)^{1-\frac{q}{N}}\|\nabla f\|_{\mathcal{M}^{p,q}}^{\frac{q}{N}}.
\end{equation}
\end{Corollaire}
%%%%%%%%%%%%%%%%%%%%%%%%%%%%%%%%%%%%%%%%%%%%%%%%%%%
\noindent {\bf Proof.} First note that if $\nu=\rho=N$, then the hypothesis (\ref{HypothesePointwise}) is $q<N$ and since $q=\frac{Np}{N-p}$, we easily obtain the constraint $1<p<\frac{N}{2}$. 
Next remark that instead of using the double weighted Poincaré inequality (\ref{PoincareSobolev_inequality}) to pass from (\ref{AvantPoinK}) to (\ref{ApresPoinK}), we can apply  the usual unweighted Poincaré inequality (\ref{PoincareSobolev_inequality1}) and then, following the same arguments as above we obtain the wished estimate (\ref{PointwiseUnweighted}). \hfill $\blacksquare$\\

To the best of our knowledge, the pointwise estimate (\ref{PointwiseUnweighted}) is new in the setting of stratified Lie groups, even in this unweighted framework. 
%%%%%%%%%%%%%%%%%%%%%%%%%%%%%%%%%%%%%%%%%%%%%%%%%%%
%%%%%%%%%%%%%%%%%%%%%%%%%%%%%%%%%%%%%%%%%%%%%%%%%%%
\mysection{Some functional inequalities}\label{Secc_FuncIneq}
We will now deduce two functional inequalities from the pointwise estimate (\ref{PointWiseIneqFeli1}). Throughout this entire section, we will assume the general hypotheses for the kernel $\Omega$, the weight $\varpi$, and the weight $\omega$ considered in Theorem \ref{Theo1}.\\

Our first result is the following Sobolev-type inequality:
%%%%%%%%%%%%%%%%%%%%%%%%%%%%%%%%%%%%%%%%%%%%%%%%%%%
\begin{Corollaire}[\bf Lebesgue-based Sobolev-type inequality]\label{CoroSobolevLebesgue}
	Over a stratified Lie group $\mathbb{G}$ of homogeneous dimension $N\geq 4$, consider a weighted singular integral operator $T_{\Omega, \varpi}$ as defined in expression (\ref{Def_Operator}) above, and consider a function $f:\mathbb{G}\longrightarrow \mathbb{R}$ such that $\nabla f$ belongs to a weighted Lebesgue space $L^q_\omega(\mathbb{G})$ and to a weighted Morrey space $\mathcal{M}_\omega^{p,q}(\mathbb{G})$ for some weight $\omega$ and for some parameters $1<p<q<+\infty$. Assume that the kernel $\Omega$ and the weights $\varpi$ and $\omega$ satisfy the general hypotheses of Theorem \ref{Theo1}.
	
	If we set the power index $0<\theta=\frac{q}{\rho}(\nu+1-N)<1$ and define a parameter $1<r<+\infty$ by the condition $(1-\frac{q}{\rho}(\nu+1-N))r=q$, then we have the inequality
	\begin{equation}\label{FuncIneqLebesgue}
	\|T_{\Omega, \varpi}(f)\|_{L^r_\omega}\leq C\|\nabla f\|_{L^q_\omega}^{1-\theta}\|\nabla f\|_{\mathcal{M}_\omega^{p,q}}^\theta.
	\end{equation}
\end{Corollaire}
%%%%%%%%%%%%%%%%%%%%%%%%%%%%%%%%%%%%%%%%%%%%%%%%%%%
\noindent {\bf Proof. } We start by the pointwise estimate 
$$T_{\Omega, \varpi}(f)(x)\leq C\left( \mathscr{M}_{\omega}(|\nabla f|^p)^{\frac{1}{p}}(x) \right)^{1-\frac{q}{\rho}(\nu+1-N)} \|\nabla f \|_{\mathcal{M}_{\omega}^{p,q}}^{\frac{q}{\rho}(\nu+1-N)}.$$
Then, we apply the weighted $L^r_\omega(\mathbb{G})$ norm to both sides of this inequality to obtain
$$\|T_{\Omega, \varpi}(f)\|_{L^r_\omega}\leq C\left\|\left( \mathscr{M}_{\omega}(|\nabla f|^p)^{\frac{1}{p}}\right)^{1-\frac{q}{\rho}(\nu+1-N)}\right\|_{L^r_\omega} \|\nabla f \|_{\mathcal{M}_{\omega}^{p,q}}^{\frac{q}{\rho}(\nu+1-N)},$$
and using the property (\ref{ProprietePuissanceLebesgue}), we can write
$$\|T_{\Omega, \varpi}(f)\|_{L^r_\omega}\leq C\left\|\mathscr{M}_{\omega}(|\nabla f|^p)\right\|_{L^{(\frac{1}{p}-\frac{q}{p\rho}(\nu+1-N))r}_\omega}^{\frac{1}{p}-\frac{q}{p\rho}(\nu+1-N)} \|\nabla f \|_{\mathcal{M}_{\omega}^{p,q}}^{\frac{q}{\rho}(\nu+1-N)}.$$
Since by hypothesis we have $1<p<q$ and the relationship $(1-\frac{q}{\rho}(\nu+1-N))r=q$, we have that $(\frac{1}{p}-\frac{q}{p\rho}(\nu+1-N))r=\frac{q}{p}>1$, and we can deduce that the maximal function $\mathscr{M}_\omega$ is bounded in the weighted Lebesgue space $L^{(\frac{1}{p}-\frac{q}{p\rho}(\nu+1-N))r}_\omega(\mathbb{G})$ (recall the control (\ref{BoundednessMaxWeightedFuncLebesgue}) above). From this fact we have the control
$$\|T_{\Omega, \varpi}(f)\|_{L^r_\omega}\leq C\left\||\nabla f|^p\right\|_{L^{(\frac{1}{p}-\frac{q}{p\rho}(\nu+1-N))r}_\omega}^{\frac{1}{p}-\frac{q}{p\rho}(\nu+1-N)} \|\nabla f \|_{\mathcal{M}_{\omega}^{p,q}}^{\frac{q}{\rho}(\nu+1-N)}.$$
Using again the property (\ref{ProprietePuissanceLebesgue}), we write 
$$\|T_{\Omega, \varpi}(f)\|_{L^r_\omega}\leq C\|\nabla f\|_{L^{(1-\frac{q}{\rho}(\nu+1-N))r}_\omega}^{1-\frac{q}{\rho}(\nu+1-N)} \|\nabla f \|_{\mathcal{M}_{\omega}^{p,q}}^{\frac{q}{\rho}(\nu+1-N)}.$$
Now we remark that, by hypothesis, we have $(1-\frac{q}{\rho}(\nu+1-N))r=q$, and we obtain the inequality
$$\|T_{\Omega, \varpi}(f)\|_{L^r_\omega}\leq C\|\nabla f\|_{L^{q}_\omega}^{1-\frac{q}{\rho}(\nu+1-N)} \|\nabla f \|_{\mathcal{M}_{\omega}^{p,q}}^{\frac{q}{\rho}(\nu+1-N)},$$
which we can rewrite as follows:
$$\|T_{\Omega, \varpi}(f)\|_{L^r_\omega}\leq C\|\nabla f\|_{L^{q}_\omega}^{1-\theta} \|\nabla f \|_{\mathcal{M}_{\omega}^{p,q}}^{\theta},$$
and this ends the proof of Corollary \ref{CoroSobolevLebesgue}. \hfill $\blacksquare$\\

Let us remark that this functional inequality leads quite easily to a new class of Sobolev estimates. Indeed, since we have the control $\|\nabla f \|_{\mathcal{M}_{\omega}^{p,q}}\leq \|\nabla f \|_{\mathcal{M}_{\omega}^{q,q}}$ and the identification $\|\nabla f \|_{\mathcal{M}_{\omega}^{q,q}}=\|\nabla f \|_{L^q_{\omega}}$, we obtain
\begin{eqnarray}
\|T_{\Omega, \varpi}(f)\|_{L^r_\omega}&\leq &C\|\nabla f\|_{L^{q}_\omega}^{1-\theta} \|\nabla f \|_{\mathcal{M}_{\omega}^{p,q}}^{\theta}\notag\\
&\leq &C\|\nabla f\|_{L^{q}_\omega}^{1-\theta} \|\nabla f \|_{\mathcal{M}_{\omega}^{q,q}}^{\theta}\simeq C\|\nabla f\|_{L^{q}_\omega}^{1-\theta} \|\nabla f \|_{L_{\omega}^{q}}^{\theta}\notag\\
&\leq & C\|\nabla f\|_{L^{q}_\omega},\label{SobolevLebesgueIneq}
\end{eqnarray}
and this estimate seems to be new in the context of weighted rough singular integral operators on stratified Lie groups. Note also that, in the unweighted setting, we have the upper and lower Ahlfors powers $\nu=\rho=N$, and thus the relationship $(1-\frac{q}{\rho}(\nu+1-N))r=q$ becomes $\frac{1}{r}=\frac{1}{q}-\frac{1}{N}$ with $1<q<N$, which are the classical relationships for the usual Sobolev inequalities on stratified Lie groups.\\

From the pointwise inequality (\ref{PointWiseIneqFeli1}), it is quite straightforward to deduce a wide class of functional inequalities. In the previous corollary, we used the weighted Lebesgue space $L^r_\omega$ as the base space. Now, we will consider a weighted Morrey space as the base space.

%%%%%%%%%%%%%%%%%%%%%%%%%%%%%%%%%%%%%%%%%%%%%%%%%%%
\begin{Corollaire}[\bf Morrey-based Sobolev-type inequality]\label{CoroSobolevMorrey}
	Over a stratified Lie group $\mathbb{G}$ of homogeneous dimension $N\geq 4$, consider a weighted singular integral operator $T_{\Omega, \varpi}$ as defined in the expression (\ref{Def_Operator}) above, and consider a function $f:\mathbb{G}\longrightarrow \mathbb{R}$ such that $\nabla f$ belongs to a weighted Morrey space $\mathcal{M}^{\mathfrak{p}, \mathfrak{q}}_\omega(\mathbb{G})$ and to a weighted Morrey space $\mathcal{M}_\omega^{p,q}(\mathbb{G})$ for some weight $\omega$ and for some parameters $1<\mathfrak{p}<\mathfrak{q}<+\infty$ and $1<p<q<+\infty$. Assume that the kernel $\Omega$ and the weights $\varpi$ and $\omega$ satisfy the general hypotheses of Theorem \ref{Theo1}. \\
	
	If we set the power index $0<\theta=\frac{q}{\rho}(\nu+1-N)<1$ and define two parameters $1<\mathfrak{m}<r<+\infty$ by the conditions $(1-\frac{q}{\rho}(\nu+1-N))\mathfrak{m}=\mathfrak{p}$ and $(1-\frac{q}{\rho}(\nu+1-N))r=\mathfrak{q}$, then we have the inequality
	\begin{equation}\label{FuncIneqMorrey}
	\|T_{\Omega, \varpi}(f)\|_{\mathcal{M}^{\mathfrak{m}, r}_\omega}\leq C\|\nabla f\|_{\mathcal{M}^{\mathfrak{p}, \mathfrak{q}}_\omega}^{1-\theta}\|\nabla f\|_{\mathcal{M}_\omega^{p,q}}^\theta.
	\end{equation}
\end{Corollaire}
%%%%%%%%%%%%%%%%%%%%%%%%%%%%%%%%%%%%%%%%%%%%%%%%%%%
\noindent {\bf Proof. } As before, the starting point is the pointwise inequality (\ref{PointWiseIneqFeli1}):
$$T_{\Omega, \varpi}(f)(x)\leq C\left( \mathscr{M}_{\omega}(|\nabla f|^p)^{\frac{1}{p}}(x) \right)^{1-\frac{q}{\rho}(\nu+1-N)} \|\nabla f \|_{\mathcal{M}_{\omega}^{p,q}}^{\frac{q}{\rho}(\nu+1-N)}.$$
Then, we apply the weighted $\mathcal{M}^{\mathfrak{m},r}_\omega(\mathbb{G})$ norm to both sides of this inequality to obtain
$$\|T_{\Omega, \varpi}(f)\|_{\mathcal{M}^{\mathfrak{m},r}_\omega}\leq C\left\|\left( \mathscr{M}_{\omega}(|\nabla f|^p)^{\frac{1}{p}}\right)^{1-\frac{q}{\rho}(\nu+1-N)}\right\|_{\mathcal{M}^{\mathfrak{m},r}_\omega} \|\nabla f \|_{\mathcal{M}_{\omega}^{p,q}}^{\frac{q}{\rho}(\nu+1-N)}.$$
We use now the property (\ref{ProprietePuissanceMorrey}) to write 
$$\|T_{\Omega, \varpi}(f)\|_{\mathcal{M}^{\mathfrak{m},r}_\omega}\leq C\left\|\mathscr{M}_{\omega}(|\nabla f|^p)\right\|_{\mathcal{M}^{(\frac{1}{p}-\frac{q}{p\rho}(\nu+1-N))\mathfrak{m},(\frac{1}{p}-\frac{q}{p\rho}(\nu+1-N))r}_\omega}^{\frac{1}{p}-\frac{q}{p\rho}(\nu+1-N)} \|\nabla f \|_{\mathcal{M}_{\omega}^{p,q}}^{\frac{q}{\rho}(\nu+1-N)},$$
but since we have $(\frac{1}{p}-\frac{q}{p\rho}(\nu+1-N))\mathfrak{m}>1$ as well as $(\frac{1}{p}-\frac{q}{p\rho}(\nu+1-N))r>1$ (recall that $1<\mathfrak{m}<r<+\infty$), by the boundedness of the weighted maximal function $\mathscr{M}_\omega$ on the weighted Morrey spaces $\mathcal{M}^{(\frac{1}{p}-\frac{q}{p\rho}(\nu+1-N))\mathfrak{m},(\frac{1}{p}-\frac{q}{p\rho}(\nu+1-N))r}_\omega(\mathbb{G})$ (recall (\ref{BoundednessMaxWeightedFuncMorrey}) above) we obtain
$$\|T_{\Omega, \varpi}(f)\|_{\mathcal{M}^{\mathfrak{m},r}_\omega}\leq C\left\||\nabla f|^p\right\|_{\mathcal{M}^{(\frac{1}{p}-\frac{q}{p\rho}(\nu+1-N))\mathfrak{m},(\frac{1}{p}-\frac{q}{p\rho}(\nu+1-N))r}_\omega}^{\frac{1}{p}-\frac{q}{p\rho}(\nu+1-N)} \|\nabla f \|_{\mathcal{M}_{\omega}^{p,q}}^{\frac{q}{\rho}(\nu+1-N)},$$
at this point we apply again the property (\ref{ProprietePuissanceMorrey}) and we have
$$\|T_{\Omega, \varpi}(f)\|_{\mathcal{M}^{\mathfrak{m},r}_\omega}\leq C\left\|\nabla f\right\|_{\mathcal{M}^{(1-\frac{q}{\rho}(\nu+1-N))\mathfrak{m},(1-\frac{q}{\rho}(\nu+1-N))r}_\omega}^{1-\frac{q}{\rho}(\nu+1-N)} \|\nabla f \|_{\mathcal{M}_{\omega}^{p,q}}^{\frac{q}{\rho}(\nu+1-N)}.$$
Since by hypothesis we have $(1-\frac{q}{\rho}(\nu+1-N))\mathfrak{m}=\mathfrak{p}$, $(1-\frac{q}{\rho}(\nu+1-N))r=\mathfrak{q}$ and $\theta=\frac{q}{\rho}(\nu+1-N)$, we obtain
$$\|T_{\Omega, \varpi}(f)\|_{\mathcal{M}^{\mathfrak{m},r}_\omega}\leq C\left\|\nabla f\right\|_{\mathcal{M}^{\mathfrak{p},\mathfrak{q}}_\omega}^{1-\theta} \|\nabla f \|_{\mathcal{M}_{\omega}^{p,q}}^{\theta},$$
which is the wished inequality.  \hfill $\blacksquare$\\

This previous inequality has several interesting avatars. Indeed, if we set $\mathfrak{p}=p$ and $\mathfrak{q}=q$, then we obtain
$$\|T_{\Omega, \varpi}(f)\|_{\mathcal{M}^{\mathfrak{m},r}_\omega}\leq C\left\|\nabla f\right\|_{\mathcal{M}^{\mathfrak{p},\mathfrak{q}}_\omega}^{1-\theta} \|\nabla f \|_{\mathcal{M}_{\omega}^{p,q}}^{\theta}= C\|\nabla f \|_{\mathcal{M}_{\omega}^{p,q}},$$
and this is, to the best of our knowledge, a new functional estimate (even in the Euclidean setting) when considering weighted rough singular integral operators and weighted Morrey spaces. Note also that if we set $\mathfrak{m}=r$, then we have $\mathfrak{p}=\mathfrak{q}$ and obtain
$$\|T_{\Omega, \varpi}(f)\|_{\mathcal{M}^{r,r}_\omega}\leq C\left\|\nabla f\right\|_{\mathcal{M}^{\mathfrak{q},\mathfrak{q}}_\omega}^{1-\theta} \|\nabla f \|_{\mathcal{M}_{\omega}^{p,q}}^{\theta},$$
which can be rewritten as
$$\|T_{\Omega, \varpi}(f)\|_{L^r_\omega}\leq C\left\|\nabla f\right\|_{L^{\mathfrak{q}}_\omega}^{1-\theta} \|\nabla f \|_{\mathcal{M}_{\omega}^{p,q}}^{\theta},$$
due to the identification of spaces $\mathcal{M}^{\sigma, \sigma}_\omega(\mathbb{G})=L^{\sigma}_\omega(\mathbb{G})$. This is a generalization of the inequality (\ref{FuncIneqLebesgue}), which can be obtained by setting $\mathfrak{q}=q$.\\ 

To end this section, we present two functional inequalities in the unweighted setting.
%%%%%%%%%%%%%%%%%%%%%%%%%%%%%%%%%%%%%%%%%%%%%%%%%%%
\begin{Corollaire}[\bf Unweighted Functional inequalities]\label{CoroSobolevLebesgueNoweight}
	Let $\mathbb{G}$ be a stratified Lie group of homogeneous dimension $N\geq 4$, and let $T_{\Omega}$ be a weighted singular integral operator as defined in the expression (\ref{Def_Operator}), with a kernel $\Omega$ that satisfies (\ref{TailleKernel}) and (\ref{IntegraleNulle}) with $\varpi\equiv 1$.\\
	
	Let $1<p,q,r<+\infty$ be indices such that $1<p<\frac{N}{2}$, $q=\frac{Np}{N-p}$, and $r=\frac{Nq}{N-q}$. If we consider a function $f:\mathbb{G}\longrightarrow \mathbb{R}$ such that $\nabla f\in L^q(\mathbb{G})$ and $\nabla f\in \mathcal{M}^{p,q}(\mathbb{G})$, then we have the inequality
	\begin{equation}\label{FuncIneqLebesgueNoweight}
	\|T_{\Omega}(f)\|_{L^r}\leq C\|\nabla f\|_{L^q}^{1-\frac{q}{N}}\|\nabla f\|_{\mathcal{M}^{p,q}}^{\frac{q}{N}},
	\end{equation}
	from which we easily deduce the following Sobolev-type estimate:
	\begin{equation}\label{FuncIneqLebesgueNoweight1}
	\|T_{\Omega}(f)\|_{L^r}\leq C\|\nabla f\|_{L^q},
	\end{equation}
	where $\frac{N}{N-1}<q<N$ and $r=\frac{Nq}{N-q}$.
\end{Corollaire}
%%%%%%%%%%%%%%%%%%%%%%%%%%%%%%%%%%%%%%%%%%%%%%%%%%%
Using the pointwise estimate (\ref{PointwiseUnweighted}) and the fact that the maximal function $\mathscr{M}$ is bounded on the Lebesgue spaces $L^\rho(\mathbb{G})$ for $\rho>1$, it is straightforward to deduce the inequality (\ref{FuncIneqLebesgueNoweight}). Since we have the space inclusion $L^q(\mathbb{G})= \mathcal{M}^{q,q}(\mathbb{G})\subset\mathcal{M}^{p,q}(\mathbb{G})$, we can easily deduce the estimate (\ref{FuncIneqLebesgueNoweight1}) from (\ref{FuncIneqLebesgueNoweight}). \\

As pointed out previously, many other (unweighted) functional inequalities can be obtained from the pointwise control (\ref{PointwiseUnweighted}), and we do not claim to be exhaustive here. 
%%%%%%%%%%%%%%%%%%%%%%%%%%%%%%%%%%%%%%%%%%%%%%%%%%%
%%%%%%%%%%%%%%%%%%%%%%%%%%%%%%%%%%%%%%%%%%%%%%%%%%%
\mysection{The stationary Navier-Stokes equations on the Heisenberg group}\label{Secc_Appl}

In this section, we give a first example of an interesting application of the functional estimates obtained in the previous sections in the general setting of stratified Lie groups. Specifically, we will explore some properties of weak solutions to the incompressible stationary Navier-Stokes equations (with a rough drift) on the Heisenberg group $\mathbb{H}$, which is one of the simplest non-trivial stratified Lie groups.\\

Before continuing, it is worth mentioning that the tools developed in the previous sections could also be applied to a wide range of partial differential equations with a rough drift to study interesting properties, such as the existence, uniqueness, and regularity of solutions in the setting of stratified Lie groups.\\ 

Let us recall the structure of this group $\mathbb{H}$. For any $x=(x_{1},x_{2},x_{3}), \, y=(y_1,y_2,y_3)\in \mathbb{R}^{3}$,  the law group is given by 
$$x\cdot y=(x_{1},x_{2},x_{3})\cdot(y_{1},y_{2},y_{3})=\left(x_{1}+y_{1},x_{2}+y_{2},x_{3}+y_{3}+\frac{1}{2}(x_{1}y_{2}-y_{1}x_{2})\right),$$
and the associated dilation is 
\begin{eqnarray*}
\delta_{\alpha}: \mathbb{R}^{3}&\longrightarrow &\mathbb{R}^{3}\\
x=(x_{1},x_{2},x_{3}) &\longmapsto & \delta_{\alpha}[x]=(\alpha x_{1}, \alpha x_{2},\alpha^{2} x_{3}).
\end{eqnarray*}
Note that the topological dimension is $n=3$ but the \emph{homogeneous} dimension is $N=4$. We consider now the norm 
$$|x|=\left[\left(x_{1}^{2}+x_{2}^{2}\right)^{2}+16 x_{3}^{2}\right]^{1/4},$$
and we will denote by $\mathbb{H}$ the structure given by $(\mathbb{R}^3,\cdot,  \delta,\, |\cdot|)$. \\

The left invariant vector fields are given by 
$$X_{1}=\frac{\partial}{\partial x_{1}}-\frac{1}{2}x_{2}\frac{\partial}{\partial x_{3}}, \qquad X_{2}=\frac{\partial}{\partial x_{2}}+\frac{1}{2} x_{1} \frac{\partial}{\partial x_{3}}\, ,\qquad\mbox{and} \qquad T=\frac{\partial}{\partial x_{3}},$$
and we note that they satisfy the property $[X_{1},X_{2}]=X_{1}X_{2}-X_{2}X_{1}= T$, so we can consider as an H\"ormander family ${\bf X}$ the two first vector fields $X_1$ and $X_2$. Now, for a regular enough function $\varphi:\mathbb{H}\longrightarrow \mathbb{R}$, its gradient is then defined as the 2D vector:
\begin{equation}\label{Def_Gradient}
\nabla \varphi=(X_1\varphi, X_2\varphi),
\end{equation}
and we define the \emph{divergence} of a vector field $\vec{\varphi}:\mathbb{H}\longrightarrow \mathbb{R}^2$, (\emph{i.e.} we have $\vec{\varphi}=(\varphi_1, \varphi_2)$ where $\varphi_{1,2}:\mathbb{H}\longrightarrow \mathbb{R}$) by the expression
\begin{equation}\label{Def_Divergence}
div(\vec{\varphi})=X_1\varphi_1+X_2\varphi_2.
\end{equation}
For the Laplace operator $\Delta$ we write $\Delta = X^2_1+X^2_2$ and for a vector field $\vec{\varphi}:\mathbb{H}\longrightarrow \mathbb{R}^2$ we have 
$$\Delta \vec{\varphi}=(\Delta \varphi_1, \Delta \varphi_2).$$
Remark that we have the vectorial identity $div(\nabla \varphi)=\Delta \varphi$. For more details on the Heisenberg group see the book \cite{Stein2}.\\

We now have enough material to present the stationary, incompressible, rough Navier-Stokes equations on the Heisenberg group $\mathbb{H}$. Let $\vu=(u_1, u_2):\mathbb{H}\longrightarrow \mathbb{R}^2$ be a vector field, and let $T_{\Omega}$ be a rough singular integral operator as considered in the expression (\ref{Def_Operator}) above (with $\varpi\equiv 1$). If $\pi:\mathbb{H}\longrightarrow \mathbb{R}$ is a scalar function, we consider the following rough version (\emph{i.e.}, with a rough drift) of the incompressible Navier-Stokes equations:
\begin{equation}\label{Eq_NS}
\Delta \vu - (T_{\Omega}(\vu)\cdot \nabla)\vu - \nabla \pi=0, 
\end{equation}
with $T_{\Omega}(\vu)=(T_{\Omega}(u_1), T_{\Omega}(u_2))$ and where the incompressibility condition is given by the divergence-free property of the rough drift $T_{\Omega}(\vu)$:
\begin{equation}\label{DivNS}
div(T_{\Omega}(\vu))=X_1T_{\Omega}(u_1)+X_2T_{\Omega}(u_2)=0.
\end{equation}
We will also impose the following two divergence-free conditions:
\begin{equation}\label{DivNS1}
div(\vu)=X_1(u_1)+X_2(u_2)=0,
\end{equation}
and
\begin{equation}\label{DivNS2}
div(\Delta\vu)=X_1(\Delta u_1)+X_2(\Delta u_2)=0.
\end{equation}
\begin{Remarque}
We remark that these two latter conditions are not equivalent, since the operators $X_j$ and $\Delta$ do not commute, and we have $div(\Delta \vu)\neq \Delta div(\vu)$.
\end{Remarque}	

Note now that the system (\ref{Eq_NS}) can be rewritten in the following expanded form
\begin{equation*}
{\left[\begin{matrix}
 {\Delta u_{1}}\\[3mm]
 {\Delta u_{2}}
\end{matrix}\right]}-
{\left[\begin{matrix}
{T_{\Omega}(u_{1})X_1(u_{1})+T_{\Omega}(u_{2})X_2(u_{1})}\\[3mm]
{T_{\Omega}(u_{1})X_1(u_{2})+T_{\Omega}(u_{2})X_2(u_{2})}
\end{matrix}\right]}
-{ \left[\begin{matrix}
{ X_1(\pi)}\\[3mm]
{X_2(\pi)}
\end{matrix}\right]}=0,\\[2mm]
\end{equation*}
remark moreover that due to the divergence free condition (\ref{DivNS}), we have the following identity $$(T_{\Omega}(\vu)\cdot \nabla)\vu=div(\vu\otimes T_{\Omega}(\vu)),$$
where $\vu\otimes T_{\Omega}(\vu)=\left[\begin{matrix}
u_{1}T_{\Omega}(u_{1}) & u_{1}T_{\Omega}(u_{2}) \\[3mm]
u_{2}T_{\Omega}(u_{1}) & u_{2}T_{\Omega}(u_{2}) \\[3mm]
\end{matrix}\right]
$ is the tensor formed by the vectors $\vu$ and $T_{\Omega}(\vu)$ and where we have 
$div(\vu\otimes T_{\Omega}(\vu))=\left[\begin{matrix}
X_1(u_{1}T_{\Omega}(u_{1})) + X_2(u_{1}T_{\Omega}(u_{2})) \\[3mm]
X_1(u_{2}T_{\Omega}(u_{1})) + X_2( u_{2}T_{\Omega}(u_{2}))\\[3mm]
\end{matrix}\right]$. See the book \cite{Chamorro_Livre} for more details about the Navier-Stokes equations.\\

%%%%%%%%%%%%%%%%%%%%%%%%%%%%%%%%%%%%%%%%%%%%%%%%%%%
We now turn our attention to the problem of uniqueness of weak solutions to the rough, incompressible Navier-Stokes equations (\ref{Eq_NS}). Note that, due to the structure of the equation, it is quite natural to consider a solution $\vu\in \dot{H}^1$ (see \cite[Chapitre 9]{Chamorro_Livre}). However, even in the classical three-dimensional Euclidean setting, uniqueness of weak $\dot{H}^1$-solutions remains a challenging open problem in general. See \cite[Chapter X]{Galdi} or \cite{Seregin} and the references therein for more details.\\

In order to obtain some uniqueness results, additional hypotheses are imposed on the solution $\vu\in \dot{H}^1$. These hypotheses generally require that $\vu\in E$, where $E$ is a suitable functional space (see \emph{e.g.} \cite{chamorro2021some}), and it is then deduced that the trivial solution (\emph{i.e.}, $\vu \equiv 0$) is unique. Note that these types of results are known in the literature as \emph{Liouville}-type theorems for the stationary Navier-Stokes equations.\\

In our next result, we will study the uniqueness of the trivial solution in the setting of the rough, incompressible Navier-Stokes equations (\ref{Eq_NS}) on the Heisenberg group, which has topological dimension $n=3$.

%%%%%%%%%%%%%%%%%%%%%%%%%%%%%%%%%%%%%%%%%%%%%%%%%%%
\begin{Theoreme}\label{TheoUniqueness}
	Over the Heisenberg group $\mathbb{H}$, consider a weighted singular integral operator $T_{\Omega}$ as defined in the expression (\ref{Def_Operator}) above, with $\varpi=1$, and where the kernel $\Omega$ satisfies the assumptions of Theorem \ref{Theo1}. Assume that $\vu\in \dot{H}^1(\mathbb{H})$ is a regular weak solution of the Navier-Stokes equation (\ref{Eq_NS}) such that the divergence-free conditions (\ref{DivNS}), (\ref{DivNS1}), and (\ref{DivNS2}) hold and such that $|\vu(x)|\underset{|x|\to +\infty}{\longrightarrow} 0$. Then $\vu \equiv 0$. 
\end{Theoreme}
%%%%%%%%%%%%%%%%%%%%%%%%%%%%%%%%%%%%%%%%%%%%%%%%%%%
\noindent Remark in particular that, in sharp contrast with the three-dimensional Euclidean case, no additional hypothesis on $\vu$ is required to deduce the uniqueness of the trivial solution to the equations (\ref{Eq_NS}). Note next that several divergence-free hypotheses (namely, (\ref{DivNS}), (\ref{DivNS1}), and (\ref{DivNS2})) are needed since the operators considered here do not commute.  We believe that, in future research, it would be interesting to study this type of result by relaxing the hypotheses (\ref{DivNS}) and (\ref{DivNS2}), which, to the best of our knowledge, is not straightforward.\\ 

Now, let us stress that, in the presence of a rough drift $T_\Omega(\vu)$, the results developed earlier (see \emph{e.g.}, the estimates of the form (\ref{FuncIneqLebesgueNoweight1})) will be crucial for our computations. Finally, we remark that the previous result also holds if one considers the normal drift $\vu$ instead of the rough drift $T_\Omega(\vu)$. Let us mention that, to the best of our knowledge, this result seems to be new in the setting of the Heisenberg group, even when no rough drift is considered.\\

\noindent {\bf Proof.}  We need first to deduce some information over the pressure $\pi$. Indeed, since we are working with a vector field $\vu\in \dot{H^1}(\mathbb{H})$, we claim that $\pi\in L^2(\mathbb{H})$. To prove this fact we apply the divergence operator to the equation (\ref{Eq_NS}), to obtain
$$div(\Delta \vu) - div((T_{\Omega}(\vu)\cdot \nabla)\vu) - div(\nabla \pi)=0.$$
By the divergence free condition (\ref{DivNS2}), we have $div(\Delta \vu)=0$ and we can write 
$$- div((T_{\Omega}(\vu)\cdot \nabla)\vu) - \Delta\pi=0,$$
from which we deduce the equation
$$\pi=(-\Delta)^{-1}div((T_{\Omega}(\vu)\cdot \nabla)\vu).$$
Now taking the $L^2$ norm in the expression above, we have 
$$\|\pi\|_{L^2}=\|(-\Delta)^{-\frac{1}{2}}(-\Delta)^{-\frac{1}{2}}div((T_{\Omega}(\vu)\cdot \nabla)\vu)\|_{L^2}\leq C\|(-\Delta)^{-\frac{1}{2}}div((T_{\Omega}(\vu)\cdot \nabla)\vu)\|_{L^{\frac{4}{3}}},$$
where in the last estimate we used the Hardy-Littlewood-Sobolev inequality given in (\ref{HLSobolevIneq}) above.\\

At this point we recall that the Riesz transforms $(-\Delta)^{-\frac{1}{2}}X_j$ are bounded in Lebesgue spaces $L^p(\mathbb{H})$ for $1<p<+\infty$ (see the estimate (\ref{RieszTransforms}) above), and we can write
$$\|\pi\|_{L^2}\leq C\|(T_{\Omega}(\vu)\cdot \nabla)\vu\|_{L^{\frac{4}{3}}}.$$
Now, by the H\"older inequality $\frac{3}{4}=\frac{1}{4}+\frac{1}{2}$, we obtain
$$\|\pi\|_{L^2}\leq C\|T_{\Omega}(\vu)\|_{L^4}\|\nabla\vu\|_{L^{2}}=C\|T_{\Omega}(\vu)\|_{L^4}\|\vu\|_{\dot{H}^1}.$$
Since we have, by (\ref{FuncIneqLebesgueNoweight1}), the control $\|T_{\Omega}(\vu)\|_{L^4}\leq C\|\vu\|_{\dot{H}^1}$, we finally obtain
$$\|\pi\|_{L^2}\leq C\|T_{\Omega}(\vu)\|_{L^4}\|\vu\|_{\dot{H}^1}\leq C\|\vu\|_{\dot{H}^1}\|\vu\|_{\dot{H}^1}<+\infty,$$
from which we deduce as claimed that $\pi\in L^2(\mathbb{H})$.\\

To continue, we remark now that since $\pi\in  L^2(\mathbb{H})$, then we have $\nabla \pi\in \dot{H}^{-1}(\mathbb{H})$ and since $\vu\in \dot{H}^1(\mathbb{H})$ the quantity $\displaystyle{\int_{\mathbb{H}}\nabla \pi \cdot \vu \,dx}$ is meaningful since we can write, by the $\dot{H}^{-1}-\dot{H}^{1}$ duality,
$$\left|\int_{\mathbb{H}}\nabla \pi \cdot \vu\, dx\right|\leq \|\nabla \pi\|_{\dot{H}^{-1}}\|\vu\|_{\dot{H}^1}<+\infty.$$
Note also that, since $\vu\in \dot{H}^{1}(\mathbb{H})$, we have $\Delta \vu \in \dot{H}^{-1}(\mathbb{H})$ and thus the quantity $\displaystyle{\int_{\mathbb{H}}\Delta \vu\cdot \vu \,dx}$ is also meaningful since 
$$\left|\int_{\mathbb{H}}\Delta \vu\cdot \vu \, dx \right|\leq \|\vu\|_{\dot{H}^1}\|\vu\|_{\dot{H}^1}<+\infty.$$
Remark finally that since $\vu\in \dot{H}^{1}(\mathbb{H})$, due to the Sobolev inequalities (\ref{HLSobolevIneq}) we have $\vu \in L^4(\mathbb{H})$ and by the estimate (\ref{FuncIneqLebesgueNoweight1}) we have also the control $\|T_\Omega(\vu)\|_{L^4}\leq C\|\vu\|_{\dot{H}^1}$. With all this information at hand we notice that the quantity $\displaystyle{\int_{\mathbb{H}} [(T_{\Omega}(\vu)\cdot \nabla)\vu]\cdot \vu\, dx}$ is also meaningful as we have by the H\"older inequality with $1=\frac{1}{4}+\frac{1}{2}+\frac{1}{4}$, 
$$\left|\int_{\mathbb{H}} [(T_{\Omega}(\vu)\cdot \nabla)\vu]\cdot \vu\, dx\right|\leq \|T_\Omega(\vu)\|_{L^4}\|\nabla \otimes \vu\|_{L^2}\|\vu\|_{L^4}\leq C\|\vu\|_{\dot{H}^1}\|\vu\|_{\dot{H}^1}\|\vu\|_{\dot{H}^1}<+\infty.$$
With all these observations, we can multiply the equation (\ref{Eq_NS}) by $\vu$ and we can integrate to obtain
\begin{equation}\label{EquationNSmulti}
\int_{\mathbb{H}}\big[\Delta \vu - (T_{\Omega}(\vu)\cdot \nabla)\vu - \nabla \pi\big]\cdot \vu\, dx=0.
\end{equation}
Since we have the divergence free property (\ref{DivNS1}), we obtain
$$\int_{\mathbb{H}} \nabla \pi \cdot \vu\,dx =-  \int_{\mathbb{H}} \pi \,div(\vu)\,dx=0.$$
Now, since by the property (\ref{DivNS}) we have the identity
$$\int_{\mathbb{H}} [(T_{\Omega}(\vu)\cdot \nabla)\vu]\cdot \vu\, dx=-\int_{\mathbb{H}} [(T_{\Omega}(\vu)\cdot \nabla)\vu]\cdot \vu\, dx,$$
and we can thus obtain that $\displaystyle{\int_{\mathbb{H}} [(T_{\Omega}(\vu)\cdot \nabla)\vu]\cdot \vu\, dx=0}$.\\

With these two facts at hand, from (\ref{EquationNSmulti}) we have
$$\int_{\mathbb{H}}\Delta \vu \cdot \vu \, dx=0,$$
from which we deduce, by an integration by parts, that
$$\int_{\mathbb{H}}|\nabla \otimes\vu|^2\, dx=0.$$
Recalling that we have the condition $|\vu(x)|\underset{|x|\to +\infty}{\longrightarrow 0}$, we finally obtain that $\vu\equiv 0$.\\

We have proven that, in the previous setting, all the solutions $\vu$ of the stationary rough Navier-Stokes equations (\ref{Eq_NS}) over the Heisenberg group that belong to the space $\dot{H}^1(\mathbb{H})$ are null, from which we deduce the uniqueness of the trivial solution in this framework. The proof of the Theorem \ref{TheoUniqueness} is now finished. \hfill $\blacksquare$\\

%%%%%%%%%%%%%%%%%%%%%%%%%%%%%%%%%%%%%%%%%%%%%%%%%%%
%%%%%%%%%%%%%%%%%%%%%%%%%%%%%%%%%%%%%%%%%%%%%%%%%%%
%%%%%%%%%%%%%%%%%%%%%%%%%%%%%%%%%%%%%%%%%%%%%%%%%%%
%%%%%%%%%%%%%%%%%%%%%%%%%%%%%%%%%%%%%%%%%%%%%%%%%%%
%\noindent {\bf Acknowledgment.} 
%The authors thanks to the referee for careful reading of the paper and for useful suggestions. 

\noindent {\bf Conflict of interest.} We declare that we do not have any commercial or associative interest that represents a conflict of interest in connection with the work submitted.\\

%%%%%%%%%%%%%%%%%%%%%%%%%%%%%%%%%%%%%%%%%%%%%%%%%%%

\end{document}